\documentclass{article}

\usepackage{
    amsfonts,
    amsmath,
    amssymb,
    amsthm,
    accents, % needed for \ubar, must be after ams packages.
    enumitem,
    float,
    graphicx,
    mathrsfs,
    shuffle,
    stmaryrd,
    subcaption,
    tikz,
    ytableau
}

\usepackage[margin=2cm]{geometry}

\newtheorem{thm}{Theorem}[section]
\newtheorem{pro}[thm]{Proposition}
\newtheorem{crl}[thm]{Corollary}
\newtheorem{lem}[thm]{Lemma}

\theoremstyle{definition}

\renewcommand{\operatorname}{\mathsf}

\title{A note on the submonoids of the tied-symmetric monoid\\containing the symmetric group}

\author{
    Diego Arcis\footnote{Departamento de Matem\'aticas, Universidad de La Serena, Cisternas 1200 -- 1700000 La Serena, Chile (\texttt{diego.arcis@userena.cl}) - ORCID: \href{https://orcid.org/0000-0003-0027-9323}{0000-0003-0027-9323}.}
}

\date{}

\usepackage{hyperref}

\newcommand{\sym}{\mathfrak{S}}

\newcommand{\B}{\mathcal{B}}
\newcommand{\E}{\mathcal{E}}
\newcommand{\J}{\mathcal{J}}
\newcommand{\U}{\mathcal{U}}

\newcommand{\BR}{\mathcal{BR}}

\renewcommand{\P}{\mathcal{P}}

\newcommand{\CC}{\mathbb{C}}

\newcommand{\Par}{\operatorname{Par}}

\renewcommand{\max}{\operatorname{max}}

\newcommand{\aspdf}{1}
\newcommand{\vcdraw}[1]{\vcenter{\hbox{#1}}}

\usetikzlibrary{calc}
\usetikzlibrary{positioning}

\newcommand{\figfou}{
	\centering
	\ifnum\aspdf=1$\begin{array}{cccccccccccc}
		\vcdraw{\includegraphics[scale=0.7]{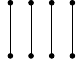}}&
		\vcdraw{\includegraphics[scale=0.7]{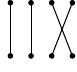}}&
		\vcdraw{\includegraphics[scale=0.7]{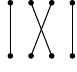}}&
		\vcdraw{\includegraphics[scale=0.7]{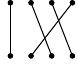}}&
		\vcdraw{\includegraphics[scale=0.7]{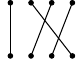}}&
		\vcdraw{\includegraphics[scale=0.7]{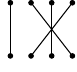}}&
		\vcdraw{\includegraphics[scale=0.7]{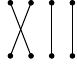}}&
		\vcdraw{\includegraphics[scale=0.7]{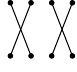}}&
		\vcdraw{\includegraphics[scale=0.7]{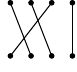}}&
		\vcdraw{\includegraphics[scale=0.7]{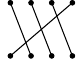}}&
		\vcdraw{\includegraphics[scale=0.7]{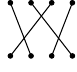}}&
		\vcdraw{\includegraphics[scale=0.7]{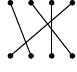}}\\
		\vcdraw{\includegraphics[scale=0.7]{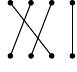}}&
		\vcdraw{\includegraphics[scale=0.7]{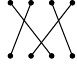}}&
		\vcdraw{\includegraphics[scale=0.7]{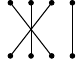}}&
		\vcdraw{\includegraphics[scale=0.7]{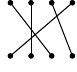}}&
		\vcdraw{\includegraphics[scale=0.7]{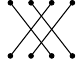}}&
		\vcdraw{\includegraphics[scale=0.7]{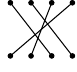}}&
		\vcdraw{\includegraphics[scale=0.7]{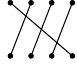}}&
		\vcdraw{\includegraphics[scale=0.7]{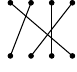}}&
		\vcdraw{\includegraphics[scale=0.7]{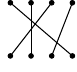}}&
		\vcdraw{\includegraphics[scale=0.7]{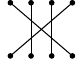}}&
		\vcdraw{\includegraphics[scale=0.7]{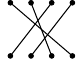}}&
		\vcdraw{\includegraphics[scale=0.7]{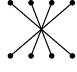}}\\
		\vcdraw{\includegraphics[scale=0.7]{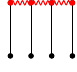}}&
		\vcdraw{\includegraphics[scale=0.7]{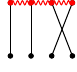}}&
		\vcdraw{\includegraphics[scale=0.7]{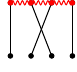}}&
		\vcdraw{\includegraphics[scale=0.7]{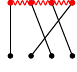}}&
		\vcdraw{\includegraphics[scale=0.7]{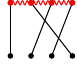}}&
		\vcdraw{\includegraphics[scale=0.7]{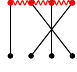}}&
		\vcdraw{\includegraphics[scale=0.7]{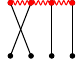}}&
		\vcdraw{\includegraphics[scale=0.7]{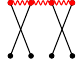}}&
		\vcdraw{\includegraphics[scale=0.7]{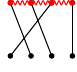}}&
		\vcdraw{\includegraphics[scale=0.7]{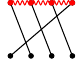}}&
		\vcdraw{\includegraphics[scale=0.7]{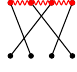}}&
		\vcdraw{\includegraphics[scale=0.7]{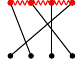}}\\
		\vcdraw{\includegraphics[scale=0.7]{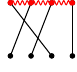}}&
		\vcdraw{\includegraphics[scale=0.7]{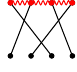}}&
		\vcdraw{\includegraphics[scale=0.7]{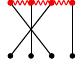}}&
		\vcdraw{\includegraphics[scale=0.7]{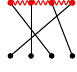}}&
		\vcdraw{\includegraphics[scale=0.7]{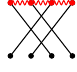}}&
		\vcdraw{\includegraphics[scale=0.7]{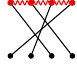}}&
		\vcdraw{\includegraphics[scale=0.7]{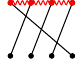}}&
		\vcdraw{\includegraphics[scale=0.7]{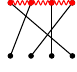}}&
		\vcdraw{\includegraphics[scale=0.7]{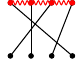}}&
		\vcdraw{\includegraphics[scale=0.7]{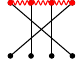}}&
		\vcdraw{\includegraphics[scale=0.7]{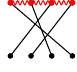}}&
		\vcdraw{\includegraphics[scale=0.7]{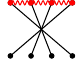}}
	\end{array}$\else
		\permutation[tkzpic=1,type=0]{1,2,3,4}
		\permutation[tkzpic=1,type=0]{1,2,4,3}
		\permutation[tkzpic=1,type=0]{1,3,2,4}
		\permutation[tkzpic=1,type=0]{1,3,4,2}
		\permutation[tkzpic=1,type=0]{1,4,2,3}
		\permutation[tkzpic=1,type=0]{1,4,3,2}
		\permutation[tkzpic=1,type=0]{2,1,3,4}
		\permutation[tkzpic=1,type=0]{2,1,4,3}
		\permutation[tkzpic=1,type=0]{2,3,1,4}
		\permutation[tkzpic=1,type=0]{2,3,4,1}
		\permutation[tkzpic=1,type=0]{2,4,1,3}
		\permutation[tkzpic=1,type=0]{2,4,3,1}
		\permutation[tkzpic=1,type=0]{3,1,2,4}
		\permutation[tkzpic=1,type=0]{3,1,4,2}
		\permutation[tkzpic=1,type=0]{3,2,1,4}
		\permutation[tkzpic=1,type=0]{3,2,4,1}
		\permutation[tkzpic=1,type=0]{3,4,1,2}
		\permutation[tkzpic=1,type=0]{3,4,2,1}
		\permutation[tkzpic=1,type=0]{4,1,2,3}
		\permutation[tkzpic=1,type=0]{4,1,3,2}
		\permutation[tkzpic=1,type=0]{4,2,1,3}
		\permutation[tkzpic=1,type=0]{4,2,3,1}
		\permutation[tkzpic=1,type=0]{4,3,1,2}
		\permutation[tkzpic=1,type=0]{4,3,2,1}
		\begin{tikzpicture}
			\permutation[tkzpic=0,type=0]{1,2,3,4}
			\tie[snake=true,snakelen=2.3,style=solid]{{1,1},{2,1},{3,1},{4,1}}
		\end{tikzpicture}\begin{tikzpicture}
			\permutation[tkzpic=0,type=0]{1,2,4,3}
			\tie[snake=true,snakelen=2.3,style=solid]{{1,1},{2,1},{3,1},{4,1}}
		\end{tikzpicture}\begin{tikzpicture}
			\permutation[tkzpic=0,type=0]{1,3,2,4}
			\tie[snake=true,snakelen=2.3,style=solid]{{1,1},{2,1},{3,1},{4,1}}
		\end{tikzpicture}\begin{tikzpicture}
			\permutation[tkzpic=0,type=0]{1,3,4,2}
			\tie[snake=true,snakelen=2.3,style=solid]{{1,1},{2,1},{3,1},{4,1}}
		\end{tikzpicture}\begin{tikzpicture}
			\permutation[tkzpic=0,type=0]{1,4,2,3}
			\tie[snake=true,snakelen=2.3,style=solid]{{1,1},{2,1},{3,1},{4,1}}
		\end{tikzpicture}\begin{tikzpicture}
			\permutation[tkzpic=0,type=0]{1,4,3,2}
			\tie[snake=true,snakelen=2.3,style=solid]{{1,1},{2,1},{3,1},{4,1}}
		\end{tikzpicture}\begin{tikzpicture}
			\permutation[tkzpic=0,type=0]{2,1,3,4}
			\tie[snake=true,snakelen=2.3,style=solid]{{1,1},{2,1},{3,1},{4,1}}
		\end{tikzpicture}\begin{tikzpicture}
			\permutation[tkzpic=0,type=0]{2,1,4,3}
			\tie[snake=true,snakelen=2.3,style=solid]{{1,1},{2,1},{3,1},{4,1}}
		\end{tikzpicture}\begin{tikzpicture}
			\permutation[tkzpic=0,type=0]{2,3,1,4}
			\tie[snake=true,snakelen=2.3,style=solid]{{1,1},{2,1},{3,1},{4,1}}
		\end{tikzpicture}\begin{tikzpicture}
			\permutation[tkzpic=0,type=0]{2,3,4,1}
			\tie[snake=true,snakelen=2.3,style=solid]{{1,1},{2,1},{3,1},{4,1}}
		\end{tikzpicture}\begin{tikzpicture}
			\permutation[tkzpic=0,type=0]{2,4,1,3}
			\tie[snake=true,snakelen=2.3,style=solid]{{1,1},{2,1},{3,1},{4,1}}
		\end{tikzpicture}\begin{tikzpicture}
			\permutation[tkzpic=0,type=0]{2,4,3,1}
			\tie[snake=true,snakelen=2.3,style=solid]{{1,1},{2,1},{3,1},{4,1}}
		\end{tikzpicture}\begin{tikzpicture}
			\permutation[tkzpic=0,type=0]{3,1,2,4}
			\tie[snake=true,snakelen=2.3,style=solid]{{1,1},{2,1},{3,1},{4,1}}
		\end{tikzpicture}\begin{tikzpicture}
			\permutation[tkzpic=0,type=0]{3,1,4,2}
			\tie[snake=true,snakelen=2.3,style=solid]{{1,1},{2,1},{3,1},{4,1}}
		\end{tikzpicture}\begin{tikzpicture}
			\permutation[tkzpic=0,type=0]{3,2,1,4}
			\tie[snake=true,snakelen=2.3,style=solid]{{1,1},{2,1},{3,1},{4,1}}
		\end{tikzpicture}\begin{tikzpicture}
			\permutation[tkzpic=0,type=0]{3,2,4,1}
			\tie[snake=true,snakelen=2.3,style=solid]{{1,1},{2,1},{3,1},{4,1}}
		\end{tikzpicture}\begin{tikzpicture}
			\permutation[tkzpic=0,type=0]{3,4,1,2}
			\tie[snake=true,snakelen=2.3,style=solid]{{1,1},{2,1},{3,1},{4,1}}
		\end{tikzpicture}\begin{tikzpicture}
			\permutation[tkzpic=0,type=0]{3,4,2,1}
			\tie[snake=true,snakelen=2.3,style=solid]{{1,1},{2,1},{3,1},{4,1}}
		\end{tikzpicture}\begin{tikzpicture}
			\permutation[tkzpic=0,type=0]{4,1,2,3}
			\tie[snake=true,snakelen=2.3,style=solid]{{1,1},{2,1},{3,1},{4,1}}
		\end{tikzpicture}\begin{tikzpicture}
			\permutation[tkzpic=0,type=0]{4,1,3,2}
			\tie[snake=true,snakelen=2.3,style=solid]{{1,1},{2,1},{3,1},{4,1}}
		\end{tikzpicture}\begin{tikzpicture}
			\permutation[tkzpic=0,type=0]{4,2,1,3}
			\tie[snake=true,snakelen=2.3,style=solid]{{1,1},{2,1},{3,1},{4,1}}
		\end{tikzpicture}\begin{tikzpicture}
			\permutation[tkzpic=0,type=0]{4,2,3,1}
			\tie[snake=true,snakelen=2.3,style=solid]{{1,1},{2,1},{3,1},{4,1}}
		\end{tikzpicture}\begin{tikzpicture}
			\permutation[tkzpic=0,type=0]{4,3,1,2}
			\tie[snake=true,snakelen=2.3,style=solid]{{1,1},{2,1},{3,1},{4,1}}
		\end{tikzpicture}\begin{tikzpicture}
			\permutation[tkzpic=0,type=0]{4,3,2,1}
			\tie[snake=true,snakelen=2.3,style=solid]{{1,1},{2,1},{3,1},{4,1}}
		\end{tikzpicture}
	\fi
}

\newcommand{\figthr}{
	\centering
	\ifnum\aspdf=1
		\includegraphics[scale=0.7]{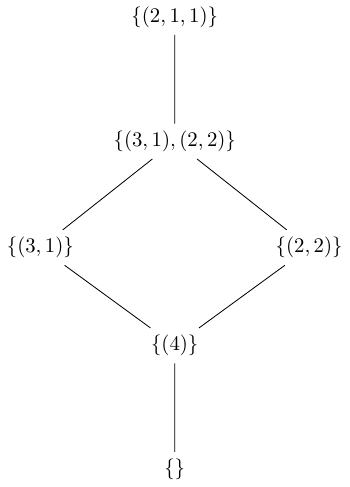}
	\else
		\begin{tikzpicture}[node distance=1.5cm]
			\node(0){$\{\}$};
			\node(4)[above=of 0]{$\{(4)\}$};
			\node(31)[above left=of 4]{$\{(3,1)\}$};
			\node(22)[above right=of 4]{$\{(2,2)\}$};
			\node(3122)[above=of $(31)!0.5!(22)$]{$\{(3,1),(2,2)\}$};
			\node(211)[above=of 3122]{$\{(2,1,1)\}$};
			\draw(0)--(4)--(31);\draw(4)--(22);\draw(31)--(3122);\draw(22)--(3122)--(211);
		\end{tikzpicture}
	\fi
}

\newcommand{\figtwo}{
	\centering
	\ifnum\aspdf=1
		\includegraphics[scale=0.7]{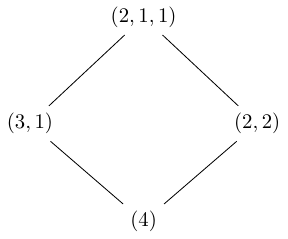}
	\else
		\begin{tikzpicture}[node distance=1.5cm]
			\node(4){$(4)$};
			\node(31)[above left=of 4]{$(3,1)$};
			\node(22)[above right=of 4]{$(2,2)$};
			\node(211)[above=of $(31)!0.5!(22)$]{$(2,1,1)$};
			\draw(4)--(31);\draw(4)--(22);\draw(31)--(211);\draw(22)--(211);
		\end{tikzpicture}
	\fi
}

\newcommand{\figone}{
	\centering
	\ifnum\aspdf=1$\begin{array}{c}
		J_{(1,1,1)}=\left\{
			\vcdraw{\includegraphics[scale=0.7]{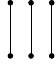}}\,,
			\vcdraw{\includegraphics[scale=0.7]{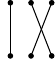}}\,,
			\vcdraw{\includegraphics[scale=0.7]{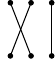}}\,,
			\vcdraw{\includegraphics[scale=0.7]{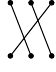}}\,,
			\vcdraw{\includegraphics[scale=0.7]{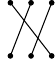}}\,,
			\vcdraw{\includegraphics[scale=0.7]{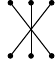}}\,
		\right\}=\sym_3\\[0.5cm]
		J_{(2,1)}=\left\{
			\vcdraw{\includegraphics[scale=0.7]{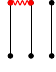}}\,,
			\vcdraw{\includegraphics[scale=0.7]{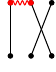}}\,,
			\vcdraw{\includegraphics[scale=0.7]{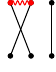}}\,,
			\vcdraw{\includegraphics[scale=0.7]{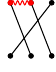}}\,,
			\vcdraw{\includegraphics[scale=0.7]{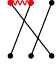}}\,,
			\vcdraw{\includegraphics[scale=0.7]{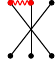}}\,,
			\vcdraw{\includegraphics[scale=0.7]{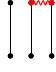}}\,,
			\vcdraw{\includegraphics[scale=0.7]{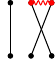}}\,,
			\vcdraw{\includegraphics[scale=0.7]{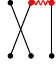}}\,,
			\vcdraw{\includegraphics[scale=0.7]{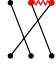}}\,,
			\vcdraw{\includegraphics[scale=0.7]{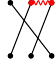}}\,,
			\vcdraw{\includegraphics[scale=0.7]{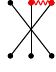}}\,,
			\vcdraw{\includegraphics[scale=0.7]{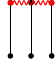}}\,,
			\vcdraw{\includegraphics[scale=0.7]{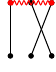}}\,,
			\vcdraw{\includegraphics[scale=0.7]{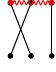}}\,,
			\vcdraw{\includegraphics[scale=0.7]{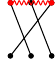}}\,,
			\vcdraw{\includegraphics[scale=0.7]{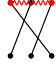}}\,,
			\vcdraw{\includegraphics[scale=0.7]{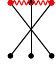}}\,
		\right\}\\[0.5cm]
		J_{(3)}=\left\{
			\vcdraw{\includegraphics[scale=0.7]{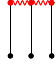}}\,,
			\vcdraw{\includegraphics[scale=0.7]{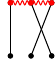}}\,,
			\vcdraw{\includegraphics[scale=0.7]{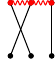}}\,,
			\vcdraw{\includegraphics[scale=0.7]{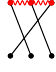}}\,,
			\vcdraw{\includegraphics[scale=0.7]{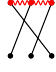}}\,,
			\vcdraw{\includegraphics[scale=0.7]{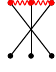}}\,
		\right\}
	\end{array}$\else
		\vpartition[type=0]{{1,-1},{2,-2},{3,-3}}
		\vpartition[type=0]{{1,-1},{2,-3},{3,-2}}
		\vpartition[type=0]{{1,-2},{2,-1},{3,-3}}
		\vpartition[type=0]{{1,-2},{2,-3},{3,-1}}
		\vpartition[type=0]{{1,-3},{2,-1},{3,-2}}
		\vpartition[type=0]{{1,-3},{2,-2},{3,-1}}
		\begin{tikzpicture}
			\vpartition[tkzpic=0,type=0]{{1,-1},{2,-2},{3,-3}}
			\tie[snake=true,snakelen=2.3,style=solid]{{1,1},{2,1}}
		\end{tikzpicture}
		\begin{tikzpicture}
			\vpartition[tkzpic=0,type=0]{{1,-1},{2,-3},{3,-2}}
			\tie[snake=true,snakelen=2.3,style=solid]{{1,1},{2,1}}
		\end{tikzpicture}
		\begin{tikzpicture}
			\vpartition[tkzpic=0,type=0]{{1,-2},{2,-1},{3,-3}}
			\tie[snake=true,snakelen=2.3,style=solid]{{1,1},{2,1}}
		\end{tikzpicture}
		\begin{tikzpicture}
			\vpartition[tkzpic=0,type=0]{{1,-2},{2,-3},{3,-1}}
			\tie[snake=true,snakelen=2.3,style=solid]{{1,1},{2,1}}
		\end{tikzpicture}
		\begin{tikzpicture}
			\vpartition[tkzpic=0,type=0]{{1,-3},{2,-1},{3,-2}}
			\tie[snake=true,snakelen=2.3,style=solid]{{1,1},{2,1}}
		\end{tikzpicture}
		\begin{tikzpicture}
			\vpartition[tkzpic=0,type=0]{{1,-3},{2,-2},{3,-1}}
			\tie[snake=true,snakelen=2.3,style=solid]{{1,1},{2,1}}
		\end{tikzpicture}
		\begin{tikzpicture}
			\vpartition[tkzpic=0,type=0]{{1,-1},{2,-2},{3,-3}}
			\tie[snake=true,snakelen=2.3,style=solid]{{2,1},{3,1}}
		\end{tikzpicture}
		\begin{tikzpicture}
			\vpartition[tkzpic=0,type=0]{{1,-1},{2,-3},{3,-2}}
			\tie[snake=true,snakelen=2.3,style=solid]{{2,1},{3,1}}
		\end{tikzpicture}
		\begin{tikzpicture}
			\vpartition[tkzpic=0,type=0]{{1,-2},{2,-1},{3,-3}}
			\tie[snake=true,snakelen=2.3,style=solid]{{2,1},{3,1}}
		\end{tikzpicture}
		\begin{tikzpicture}
			\vpartition[tkzpic=0,type=0]{{1,-2},{2,-3},{3,-1}}
			\tie[snake=true,snakelen=2.3,style=solid]{{2,1},{3,1}}
		\end{tikzpicture}
		\begin{tikzpicture}
			\vpartition[tkzpic=0,type=0]{{1,-3},{2,-1},{3,-2}}
			\tie[snake=true,snakelen=2.3,style=solid]{{2,1},{3,1}}
		\end{tikzpicture}
		\begin{tikzpicture}
			\vpartition[tkzpic=0,type=0]{{1,-3},{2,-2},{3,-1}}
			\tie[snake=true,snakelen=2.3,style=solid]{{2,1},{3,1}}
		\end{tikzpicture}
		\begin{tikzpicture}
			\vpartition[tkzpic=0,type=0]{{1,-1},{2,-2},{3,-3}}
			\tie[snake=true,snakelen=2.3,style=solid]{{1,1},{3,1}}	
		\end{tikzpicture}
		\begin{tikzpicture}
			\vpartition[tkzpic=0,type=0]{{1,-1},{2,-3},{3,-2}}
			\tie[snake=true,snakelen=2.3,style=solid]{{1,1},{3,1}}	
		\end{tikzpicture}
		\begin{tikzpicture}
			\vpartition[tkzpic=0,type=0]{{1,-2},{2,-1},{3,-3}}
			\tie[snake=true,snakelen=2.3,style=solid]{{1,1},{3,1}}	
		\end{tikzpicture}
		\begin{tikzpicture}
			\vpartition[tkzpic=0,type=0]{{1,-2},{2,-3},{3,-1}}
			\tie[snake=true,snakelen=2.3,style=solid]{{1,1},{3,1}}	
		\end{tikzpicture}
		\begin{tikzpicture}
			\vpartition[tkzpic=0,type=0]{{1,-3},{2,-1},{3,-2}}
			\tie[snake=true,snakelen=2.3,style=solid]{{1,1},{3,1}}	
		\end{tikzpicture}
		\begin{tikzpicture}
			\vpartition[tkzpic=0,type=0]{{1,-3},{2,-2},{3,-1}}
			\tie[snake=true,snakelen=2.3,style=solid]{{1,1},{3,1}}	
		\end{tikzpicture}
		\begin{tikzpicture}
			\vpartition[tkzpic=0,type=0]{{1,-1},{2,-2},{3,-3}}
			\tie[snake=true,snakelen=2.3,style=solid]{{1,1},{2,1},{3,1}}
		\end{tikzpicture}
		\begin{tikzpicture}
			\vpartition[tkzpic=0,type=0]{{1,-1},{2,-3},{3,-2}}
			\tie[snake=true,snakelen=2.3,style=solid]{{1,1},{2,1},{3,1}}
		\end{tikzpicture}
		\begin{tikzpicture}
			\vpartition[tkzpic=0,type=0]{{1,-2},{2,-1},{3,-3}}
			\tie[snake=true,snakelen=2.3,style=solid]{{1,1},{2,1},{3,1}}
		\end{tikzpicture}
		\begin{tikzpicture}
			\vpartition[tkzpic=0,type=0]{{1,-2},{2,-3},{3,-1}}
			\tie[snake=true,snakelen=2.3,style=solid]{{1,1},{2,1},{3,1}}
		\end{tikzpicture}
		\begin{tikzpicture}
			\vpartition[tkzpic=0,type=0]{{1,-3},{2,-1},{3,-2}}
			\tie[snake=true,snakelen=2.3,style=solid]{{1,1},{2,1},{3,1}}
		\end{tikzpicture}
		\begin{tikzpicture}
			\vpartition[tkzpic=0,type=0]{{1,-3},{2,-2},{3,-1}}
			\tie[snake=true,snakelen=2.3,style=solid]{{1,1},{2,1},{3,1}}
		\end{tikzpicture}
	\fi
}

\newcommand{\figzer}{
	\centering
	\ifnum\aspdf=1$
		\vcdraw{\includegraphics[scale=0.9]{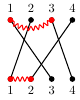}}\,\,=\,
		\vcdraw{\includegraphics[scale=0.9]{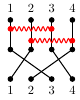}}\,\,=\,
		\vcdraw{\includegraphics[scale=0.9]{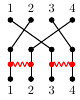}}
	$\else
		\begin{tikzpicture}
			\vpartition[tkzpic=0,type=2]{{2,-1},{4,-2},{1,-3},{3,-4}}
			\tie[snake=true,snakelen=2.3,style=solid,bend=-40]{{1,1},{3,1}}
			\tie[snake=true,snakelen=2.3,style=solid]{{1,0},{2,0}}
		\end{tikzpicture}
		\begin{tikzpicture}
			\vpartition[tkzpic=0,type=1,height=0.5,floor=1,bullb=0]{{1,-1},{2,-2},{3,-3},{4,-4}}
			\vpartition[tkzpic=0,type=-1,height=0.5]{{2,-1},{4,-2},{1,-3},{3,-4}}
			\tie[snake=true,snakelen=2.3,style=solid]{{1,0.85},{3,0.85}}
			\tie[snake=true,snakelen=2.3,style=solid]{{2,0.65},{4,0.65}}
		\end{tikzpicture}
		\begin{tikzpicture}
			\vpartition[tkzpic=0,type=1,height=0.5,floor=1,bullb=0]{{2,-1},{4,-2},{1,-3},{3,-4}}
			\vpartition[tkzpic=0,type=-1,height=0.5]{{1,-1},{2,-2},{3,-3},{4,-4}}
			\tie[snake=true,snakelen=2.3,style=solid]{{1,0.25},{2,0.25}}
			\tie[snake=true,snakelen=2.3,style=solid]{{3,0.25},{4,0.25}}
		\end{tikzpicture}
	\fi
}

\begin{document}

\maketitle

\begin{abstract}
We show that the submonoids of the tied-symmetric monoid containing the symmetric group form a distributive lattice. Furthermore, we determine a minimal generating set for any such submonoid.
\end{abstract}

\medskip{\small
}

\section{Introduction}\label{011}

Deformation algebras play a central role in the classification of algebraic structures, knot theory, and representation theory. A prominent example is the \emph{algebra of braids and ties} $\E_n(u)$, which was introduced by Aicardi and Juyumaya to construct representations of the braid group and invariants for classical and singular links~\cite{AiJu00,AiJu16b,AiJu21}. The algebraic behavior of $\E_n(u)$ is intrinsically tied to its diagrammatic foundation; indeed, it can be realized as a quotient of the monoid algebra of the so-called \emph{tied braid monoid} $T\B_n$, which extends the classical braid group $\B_n$~\cite{AiJu16,AiJu21}.

The tied-symmetric monoid $T\sym_n$ emerges naturally when specializing the parameter to $u=1$ in the algebra of braids and ties, exactly recovering its monoid algebra, $\E_n(1)=\CC[T\sym_n]$. Furthermore, by imposing the Coxeter involution relations on the standard braid generators, $T\B_n$ quotients down to $T\sym_n$~\cite{AiArJu24}. Due to its combinatorial richness, $T\sym_n$ has become an object of independent study, providing a structural bridge between knot-theoretic algebras and classical partition monoids~\cite{ArJu21,AiArJu23,AiArJu24,ArEs26,PreArJu26}.

Recently, it was established that the well-known monoid of uniform block permutations $\U_n$ can be realized as a quotient of the tied-symmetric monoid $T\sym_n$~\cite[Proposition~3.1]{PreArJu26}. Due to this, many of the algebraic and combinatorial properties of $\U_n$ can be lifted and extended to the framework of $T\sym_n$~\cite{OrSaSchZa22,OrSaSchZa25}.

To formalize this framework, for a positive integer $n$, let $[n]=\{1,\ldots,n\}$. We denote by $\sym_n$ the \emph{symmetric group} of permutations of $[n]$. As shown in \cite{Moo896}, it admits a presentation by generators $s_1,\ldots,s_{n-1}$ subject to the classical Coxeter relations, where $s_i$ denotes the standard simple transposition interchanging $i$ and $i+1$. 
A \emph{set partition} of $[n]$ is a collection of pairwise disjoint, non-empty subsets, ordered by their minimum elements and called \emph{blocks}, whose union is $[n]$. The collection of all such set partitions, denoted by $\P_n$, forms a lattice ordered by refinement, which endows it with the structure of an idempotent commutative monoid under the supremum operation.

The \emph{tied-symmetric monoid} is defined as the semidirect product $T\sym_n=\P_n\rtimes\sym_n$, where the symmetric group acts naturally on $\P_n$ by $s\cdot(I_1,\ldots,I_k)=(s(I_1),\ldots,s(I_k))$~\cite[Subsubsection~5.1.1]{AiArJu23}. The elements of $T\sym_n$ are usually represented by \emph{strand diagrams of ties}. See Figure~\ref{000}.\begin{figure}[H]
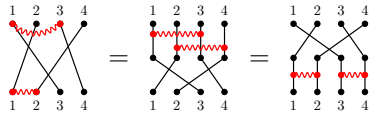
\figzer\caption{An element of the tied-symmetric monoid $T\sym_n$.}\label{000}\end{figure}%$(\{1,3\},\{2,4\})s_1s_3s_2=s_1s_3s_2(\{1,2\},\{3,4\})$

In this paper, we characterize the family of submonoids of $T\sym_n$ that contain the symmetric group $\sym_n$. Our first main result extends a result by Orellana, Saliola, Schilling, and Zabrocki, which establishes that the submonoids of $\U_n$ containing $\sym_n$ form a distributive lattice.
\begin{thm}\label{008}
The union of two submonoids of $T\sym_n$ containing $\sym_n$ is also a submonoid of $T\sym_n$. Consequently, these submonoids form a distributive lattice under the operations of union and intersection.
\end{thm}

Furthermore, we solve the problem of determining a minimal generating set for any such submonoid. The structure of these submonoids is governed by integer partitions. A \emph{partition} of $n$ is a weakly decreasing sequence of positive integers, called \emph{parts}, whose sum is $n$. The collection of partitions of $n$ is denoted by $\Par_n$. The \emph{type} of a set partition $e\in\P_n$ is the integer partition $\|e\|$ formed by the sizes of the blocks of $e$ sorted in non-increasing order. We utilize the partial order on $\Par_n$ introduced in~\cite[Definition~3.1]{OrSaSchZa25}. Given partitions $\lambda,\lambda'\in\Par_n$, we write $\lambda\preceq\lambda'$ if there exist set partitions $d_1,\ldots,d_r\in\P_n$ such that $\|d_1\cdots d_r\|=\lambda$ and $\|d_1\|=\cdots=\|d_r\|=\lambda'$. This relation defines a partial order on $\Par_n$~\cite[Proposition~3.3]{OrSaSchZa25}. We set $\Par_n^*=\Par_n\setminus\{(1,\ldots,1)\}$. A \emph{down-set} is a subset $I$ that is closed under smaller elements; that is, if $y \in I$ and $x\preceq y$, then $x\in I$. Our second main result reduces the generation problem to a finite antichain of idempotents.

\begin{thm}\label{010}
Let $I$ be a down-set of $(\Par_n^*,\preceq)$ with $\max(I)=\{\lambda_1,\ldots,\lambda_k\}$, and let $M$ be its corresponding submonoid of $T\sym_n$ containing $\sym_n$. For each $i\in[k]$, fix an idempotent $h_i \in \P_n$ such that $\|h_i\|=\lambda_i$. Then, $M$ is generated by the simple transpositions $s_1,\ldots,s_{n-1}$ together with the idempotents $h_1,\ldots,h_k$.
\end{thm}

\section{The submonoids of $T\sym_n$ containing $\sym_n$}\label{012}

In this section, we study the submonoids of $T\sym_n$ containing $\sym_n$ and prove our main results. Specifically, Theorem~\ref{008} is proven in Subsection~\ref{013}, while Theorem~\ref{010} is proven in Subsection~\ref{014}.

By definition of the tied-symmetric monoid $T\sym_n=\P_n\rtimes\sym_n$, for every $x\in T\sym_n$ there exists a unique $s\in\sym_n$ and unique $e,e'\in\P_n$ such that $x=es=se'$. In fact, these elements are related by the action, as $e'=s^{-1}\cdot e$. For instance, for the element shown in Figure~\ref{000}, we have $s=s_1s_3s_2$ with $e=(\{1,3\},\{2,4\})$ and $e'=(\{1,2\},\{3,4\})$.

\begin{pro}\label{009}
Let $M$ be a submonoid of $T\sym_n$ containing $\sym_n$, and let $\P_n^M=\P_n\cap M$ be the submonoid of idempotents of $M$. Then $M$ decomposes as the semidirect product $M=\P_n^M\rtimes\sym_n$.
\end{pro}
\begin{proof}
Since $M\subseteq T\sym_n$ and $T\sym_n=\P_n\rtimes\sym_n$, every element in $M$ can be uniquely written as $es$ for some $e\in\P_n$ and $s\in\sym_n$. Because $\sym_n\subseteq M$, we have $e=(es)s^{-1}\in M$. Thus, $e\in P_n\cap M=\P_n^M$. This proves that every element of $M$ can be uniquely factored as a product of an element in $\P_n^M$ and an element in $\sym_n$. Furthermore, given $e\in\P_n^M$ and $s\in\sym_n$, we have $s\cdot e\in\P_n^M$ because $s\cdot e\in\P_n$ and $s,e,s^{-1}\in M$. Therefore, $\P_n^M$ is invariant under the action of $\sym_n$, which allows us to conclude that $M=\P_n^M\rtimes\sym_n$.
\end{proof}

A monoid $M$ is \emph{inverse} if for every $x\in M$ there exists unique $x^*\in M$, called the \emph{inverse} of $x$, such that $xx^*x=x$ and $x^*xx^*=x^*$. As shown in~\cite[Corollary~1]{AiArJu24}, the tied-symmetric monoid $T\sym_n$ is inverse. Indeed, if $x=es\in T\sym_n$, then $x^*=s^{-1}e=e's^{-1}$, where $e'=s^{-1}\cdot e$.

\begin{pro}
Every submonoid of $T\sym_n$ containing $\sym_n$ is inverse.
\end{pro}
\begin{proof}
Let $M$ be a submonoid of $T\sym_n$ containing $\sym_n$, and let $x\in M$. Since $M\subseteq T\sym_n$, the element $x$ can be uniquely factored as $x=es$ for some $e\in\P_n$ and $s\in\sym_n$. Because $\sym_n\subseteq M$, we have $s^{-1}\in M$, and then $e=(es)s^{-1}\in M$. Hence, their product $x^*=s^{-1}e$ must also belong to $M$. Therefore $M$ is inverse.
\end{proof}

Two elements $x,y$ in a monoid $M$ are said to be \emph{$\J$-equivalent} if $MxM=MyM$. This defines an equivalence relation on $M$ whose equivalence classes are called the \emph{$\J$-classes} of $M$.

\begin{pro}[{\cite[Subsection~6.1]{PreArJu26}}]\label{003}
The $\J$-classes of $T\sym_n$ are indexed by the partitions of $n$, namely $J_\lambda$ with $\lambda\in\Par_n$. More precisely, $J_\lambda=\{s e s' \mid s,s'\in\sym_n\}$, where $e\in\P_n$ is any fixed set partition satisfying $\|e\|=\lambda$.
\end{pro}
Note that the $\J$-class indexed by the trivial partition $(1,\ldots,1)$ corresponds precisely to $\sym_n$. See Figure~\ref{001}.
\begin{figure}[H]\figone\caption{The $\J$-classes of the tied-symmetric monoid $T\sym_3$.}\label{001}\end{figure}

\subsection{Proof of Theorem~\ref{008}}\label{013}

Note that by definition of the action, we have $\|s\cdot e\|=\|e\|$ for all $s\in\sym_n$. Consequently, $\sym_n$ acts transitively on the set of all set partitions of a given type.

\begin{lem}\label{004}
Every submonoid of $T\sym_n$ containing $\sym_n$ is a union of $\J$-classes of $T\sym_n$.
\end{lem}
\begin{proof}
Let $M$ be a submonoid of $T\sym_n$ containing $\sym_n$. If a $\J$-class $J_\lambda$ of $T\sym_n$ intersects $M$, there exist $s,s'\in\sym_n$ and $e\in\P_n$ with $\|e\|=\lambda$ such that $ses'\in M$. Since $\sym_n\subseteq M$, it follows that $e\in M$, and hence $J_\lambda\subseteq M$.
\end{proof}

\begin{lem}\label{005}
Let $e,e'\in\P_n$ and let $\sym_n^e$ be the submonoid of $T\sym_n$ generated by $\sym_n\cup\{e\}$. Then:
\begin{enumerate}
\item If $\|e'\|\neq(1,\ldots,1)$, then $e'\in\sym_n^e$ if and only if $\|e'\|\preceq\|e\|$.
\item The monoid $\sym_n^e$ is the union of $\sym_n$ together with the $\J$-classes $J_\lambda$ with $\lambda\preceq\|e\|$.
\end{enumerate}
\end{lem}
\begin{proof}
For the first part, consider $e'\in\sym_n^e$ with $\|e'\|\neq(1,\ldots,1)$. Since $\sym_n^e$ is generated by $\sym_n\cup\{e\}$, any element $x\in\sym_n^e\setminus\sym_n$ can be written as an alternating product of permutations $t_0,t_1,\ldots,t_k\in\sym_n$ and the idempotent $e$:\[x=t_0et_1e\cdots t_{k-1}et_k=d_1\cdots d_k t_0t_1\cdots t_k,\quad\text{where}\quad d_i=(t_0t_1\cdots t_{i-1})\cdot e\quad\text{for all}\quad i\in[k].\]If $x=e'$, then $t_0\cdots t_k=1$, which implies that $e'=d_1\cdots d_k$. Since $\|d_1\cdots d_k\|=\|e'\|$ and $\|d_i\|=\|(t_0\cdots t_{i-1})\cdot e\|=\|e\|$ for all $i\in[k]$, we conclude that $\|e'\|\preceq\|e\|$. Conversely, assume that $\|e'\|\preceq\|e\|$, which means there exist set partitions $d_1,\ldots,d_r\in\P_n$ such that $\|d_1\|=\dots=\|d_r\|=\|e\|$ and $\|d_1\cdots d_r\|=\|e'\|$. Proposition~\ref{003} ensures that each $d_i$ belongs to the $\J$-class indexed by $\|e\|$, implying the existence of permutations $r_i,r_i'\in\sym_n$ such that $d_i=r_ier_i'$. Since $\sym_n \subseteq\sym_n^e$ and $e\in\sym_n^e$, it follows that $d_1,\ldots,d_r\in\sym_n^e$, and since $\sym_n^e$ is a submonoid, their product $d=d_1\cdots d_r$ also belongs to $\sym_n^e$. Finally, since $\|d\|=\|e'\|$, both $d$ and $e'$ reside in the $\J$-class indexed by $\|e'\|$. Thus, there exist $r,r'\in\sym_n$ such that $e'=r d r'$, which allows us to conclude that $e'\in\sym_n^e$.

For the second part, by Lemma~\ref{004}, since $\sym_n^e$ is a submonoid containing $\sym_n$, it must be a union of $\J$-classes of $T\sym_n$. Recall that $\sym_n$ corresponds to the $\J$-class indexed by the trivial partition. For any other partition $\lambda\in\Par_n$, the class $J_\lambda$ is contained in $\sym_n^e$ if and only if it contains at least one set partition $e'$ of type $\lambda$. By the first part, this holds if and only if $\lambda=\|e'\|\preceq\|e\|$. Therefore, the monoid $\sym_n^e$ is the union of $\sym_n$ together with the $\J$-classes $J_\lambda$ satisfying $\lambda\preceq\|e\|$, as required.
\end{proof}

A partition $\lambda'$ is said to be \emph{coarser} than a partition $\lambda$ if the parts of $\lambda'$ can be obtained by adding together parts of $\lambda$. As shown in \cite[Theorem~3.4]{OrSaSchZa25},\begin{equation}\label{006}\lambda\preceq\lambda'\text{ if and only if $\lambda$ is coarser than $\lambda'$ and }\varsigma\geq\varsigma',\end{equation}where $\varsigma$ and $\varsigma'$ denote the smallest parts of $\lambda$ and $\lambda'$, respectively, that are not equal to $1$. See Figure~\ref{002}.

\begin{figure}[H]
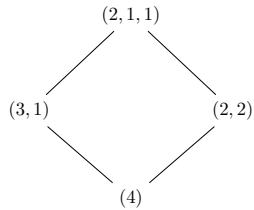
\figtwo\caption{Subposet $(\Par_n^*,\preceq)$ of non-trivial partitions of $4$.}\label{002}\end{figure}

\begin{proof}[{\bf Proof of Theorem~\ref{008}}]
Let $M_1$ and $M_2$ be two submonoids of $T\sym_n$ containing $\sym_n$. The identity element trivially belongs to $M_1\cup M_2$. To prove closure under multiplication, let $x,y\in M_1\cup M_2$. If both elements belong to the same submonoid $M_i$, their product $xy$ trivially belongs to $M_i\subseteq M_1\cup M_2$. Suppose then, without loss of generality, that $x\in M_1$ and $y\in M_2$. Let $s,s'\in\sym_n$ and $e,e'\in\P_n$ such that $x=es$ and $y=e's'$. Then $xy=ese's'=edss'$, where $d=s\cdot e'$. Because $\sym_n$ is contained in both submonoids, we have $e\in M_1$ and $e'\in M_2$, and so $\sym_n^e\subseteq M_1$ and $\sym_n^d\subseteq M_2$. If $\|e\|=(1,\ldots,1)$, then $e=1$, meaning $ed=d\in\sym_n^d$. Symmetrically, if $\|d\|=\|e'\|=(1,\ldots,1)$, then $d=1$, meaning $ed=e\in\sym_n^e$. Assume now that neither $\|e\|$ nor $\|d\|$ is the trivial partition. By the second part of Lemma~\ref{005}, $ed\in\sym_n^e\cup\sym_n^d$ if and only if $\|ed\|\preceq\|e\|$ or $\|ed\|\preceq\|d\|$. Since $ed$ is obtained by taking the supremum of $e$ and $d$, the partition $\|ed\|$ is coarser than both $\|e\|$ and $\|d\|$. Suppose for the sake of contradiction that $\|ed\|\not\preceq\|e\|$ and $\|ed\|\not\preceq\|d\|$. Since $\|ed\|$ is coarser than both, by~\eqref{006} it must be that $\varsigma<\varsigma_e$ and $\varsigma<\varsigma_d$, where $\varsigma,\varsigma_e,\varsigma_d$ are the smallest parts of the partitions $\|ed\|$, $\|e\|$ and $\|d\|$, respectively, that are not equal to $1$. By definition, $ed$ contains a block of size $\varsigma$. Since all non-singleton blocks of $e$ have sizes at least $\varsigma_e>\varsigma$, this block of size $\varsigma$ in $ed$ can only be formed by joining singleton blocks of $e$. In order to merge these singletons, they must be connected by a block in $d$. This implies that $d$ contains a block of size exactly $\varsigma$, which means that $\|d\|$ has a part equal to $\varsigma$. As a result, $\varsigma_d\leq\varsigma$, which contradicts our previous deduction that $\varsigma<\varsigma_d$. Thus, we must have $\|ed\|\preceq\|e\|$ or $\|ed\|\preceq\|d\|$, meaning that $ed\in\sym_n^e\cup\sym_n^d\subseteq M_1\cup M_2$. Finally, since $ss'\in\sym_n\subseteq M_1\cup M_2$, we conclude that $xy=(ed)(ss')\in M_1\cup M_2$.
\end{proof}

\begin{crl}
Let $M$ and $N$ be two submonoids of $T\sym_n$ containing $\sym_n$. Then $\P_n^M\cup\P_n^N$ is a submonoid of $\P_n$, and $M\cup N$ decomposes as a semidirect product $M\cup N=(\P_n^M \cup \P_n^N) \rtimes \sym_n$.
\end{crl}
\begin{proof}
By Theorem~\ref{008}, the union $M\cup N$ is a submonoid of $T\sym_n$ containing $\sym_n$. By Proposition~\ref{009}, any such submonoid decomposes as $M\cup N=\P_n^{M \cup N}\rtimes\sym_n$. By definition, we have $\P_n^{M\cup N}=(M\cup N)\cap\P_n=(M\cap\P_n)\cup(N\cap\P_n)=\P_n^M\cup\P_n^N$. Therefore, $\P_n^M\cup\P_n^N$ is a submonoid of $\P_n$, and $M\cup N=(\P_n^M\cup\P_n^N)\rtimes\sym_n$.
\end{proof}

\begin{crl}\label{007}
\begin{enumerate}
\item Let $M$ be a submonoid of $T\sym_n$ containing $\sym_n$. Then the set of partitions $I=\{\|e\|\in\Par_n^*\mid e\in\P_n^M\}$ is a down-set of the subposet $(\Par_n^*,\preceq)$, and $M=\sym_n\cup\bigsqcup_{\lambda\in I}J_\lambda$.
\item Conversely, if $I$ is a down-set of $(\Par_n^*,\preceq)$, then $\sym_n\cup\bigsqcup_{\lambda\in I}J_\lambda$ is a submonoid of $T\sym_n$.
\end{enumerate}
\end{crl}
\begin{proof}
For the first part, suppose $\lambda\preceq\|e\|$ for some $e\in \P_n^M$ with $\|e\|\in\Par_n^*$. By the first part of Lemma~\ref{005}, for any $e'\in\P_n$ with $\|e'\|=\lambda$, we have $e'\in \sym_n^e$. Since $e\in\P_n^M$ and $\sym_n\subseteq M$, it follows that $\sym_n^e\subseteq M$. Thus, $e'\in M$, which implies $e'\in \P_n^M$ and therefore $\lambda=\|e'\|\in I$. 

Next, recall that every element in $M$ can be uniquely written as $se'$ for some $s\in\sym_n$ and $e'\in\P_n$. Since $s\in\sym_n\subseteq M$, we have $s^{-1}(se')=e'\in M$, meaning that $e'\in \P_n^M$. Thus, $M$ is generated by $\sym_n$ and $\P_n^M$, which allows us to write $M=\bigcup_{e\in\P_n^M}\sym_n^e$. Applying Lemma~\ref{005}, we obtain $M=\bigcup_{e\in\P_n^M}(\sym_n\cup\bigsqcup_{\lambda\preceq\|e\|}J_\lambda)=\sym_n\cup\bigsqcup_{\lambda\in I}J_\lambda$. 

For the second part, assume that $I$ is a down-set of $(\Par_n^*,\preceq)$. For each $\lambda \in I$, fix an idempotent $e_\lambda\in\P_n$ with $\|e_\lambda\|=\lambda$. Since $I$ is a down-set, it contains every partition $\lambda'\preceq\lambda$. By Lemma~\ref{005}, we have $\sym_n^{e_\lambda}\subseteq\sym_n\cup\bigsqcup_{\lambda'\in I} J_{\lambda'}$. Consequently, since $J_\lambda\subseteq\sym_n^{e_\lambda}$ for every $\lambda\in I$, we can express this union as $\bigcup_{\lambda\in I}\sym_n^{e_\lambda}$. Since this is a union of submonoids of $T\sym_n$ containing $\sym_n$, Theorem~\ref{007} ensures that $\sym_n\cup\bigsqcup_{\lambda\in I}J_\lambda$ is also a submonoid of $T\sym_n$.
\end{proof}

For instance, the submonoid of $T\sym_4$ containing $\sym_4$, associated to the down-set $I=\{(4),(3,1),(2,2)\}$ of $(\Par_4^*,\preceq)$, is given by $\sym_4\sqcup J_{(4)}\sqcup J_{(3,1)}\sqcup J_{(2,2)}$. See Figure~\ref{002}.

\subsection{Proof of Theorem~\ref{010}}\label{014}

An \emph{antichain} of a poset is a subset of pairwise non-comparable elements. Indeed, each down-set $I$ is uniquely determined by the antichain of its maximal elements $\max(I)$ \cite[Section~3.4]{St97}. Consequently, by Corollary~\ref{007}, every submonoid of $T\sym_n$ containing $\sym_n$ is uniquely determined by an antichain of $(\Par_n^*,\preceq)$. See Figure~\ref{003}.

\begin{figure}[H]
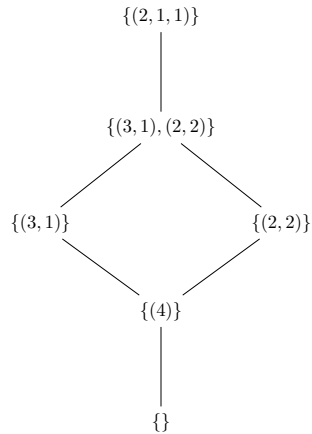
\figthr\caption{The lattice of submonoids of $T\sym_4$ containing $\sym_4$, represented by their corresponding antichain.}\label{003}\end{figure}

\begin{proof}[{\bf Proof of Theorem~\ref{010}}]
Let $S=\{s_1,\ldots,s_{n-1}\}$ denote the standard set of simple transpositions generating $\sym_n$, let $H=\{h_1,\ldots,h_k\}$, and let $N$ be the submonoid of $T\sym_n$ generated by $S\cup H$. Since $S\subseteq N$, the submonoid $N$ contains $\sym_n$. For each $i\in[k]$, we have $\|h_i\|=\lambda_i\in\max(I)\subseteq I$. By definition, $h_i\in J_{\lambda_i}$, and Corollary~\ref{007} ensures that $J_{\lambda_i}\subseteq M$. It follows that $H\subseteq M$, and consequently, $N\subseteq M$. Conversely, from Proposition~\ref{009}, any element $x\in M$ can be uniquely factored as $x=es$ for some $e\in\P_n^M$ and $s\in \sym_n$. Since $\sym_n\subseteq N$, it suffices to prove that $\P_n^M\subseteq N$. Let $e\in\P_n^M$. If $e=1$, then $e\in\sym_n\subseteq N$. Assume now that $e\neq1$, which implies that $\|e\|=\lambda\in I$. Because $I$ is uniquely determined by its antichain of maximal elements, there is some index $i\in[k]$ such that $\lambda\preceq\lambda_i$. By definition, there exist idempotents $d_1,\ldots,d_r\in\P_n$ such that $\|d_1\cdots d_r\|=\lambda$ and $\|d_j\|=\lambda_i$ for all $j\in[r]$. Let $e'=d_1\cdots d_r$. Since $\|e'\|=\lambda=\|e\|$ and the symmetric group $\sym_n$ acts transitively on the set of all set partitions of a given type, there exists a permutation $t\in\sym_n$ such that $e=t\cdot e'=te't^{-1}$. Since $\sym_n \subseteq N$, we can distribute the conjugation over the product to obtain $e=(td_1t^{-1})\cdots(td_rt^{-1})$, where $\|td_jt^{-1}\|=\lambda_i$ for all $j\in[r]$. Similarly, since $\|h_i\|=\lambda_i$, for each $j\in[r]$ there exists a permutation $u_j\in\sym_n$ such that $td_jt^{-1}=u_j\cdot h_i=u_jh_iu_j^{-1}$. Because $u_j\in\sym_n\subseteq N$ and $h_i$ belongs to $N$ by construction, each $td_jt^{-1}$ lies in $N$. Consequently, their product $e=te't^{-1}$ belongs to $N$. Hence, $x=es\in N$. This proves that $M\subseteq N$, which establishes the equality $M=N$.
\end{proof}

\begin{crl}\label{011}
Let $I$ be a down-set of $(\Par_n^*,\preceq)$ with $\max(I)=\{\lambda_1,\ldots,\lambda_k\}$, and let $M$ be its corresponding submonoid of $T\sym_n$ containing $\sym_n$. For each $i\in[k]$, fix an idempotent $h_i\in\P_n$ such that $\|h_i\|=\lambda_i$. Then, the submonoid of idempotents $\P_n^M$ is generated by the set of all conjugates of the representatives $h_i$.
\end{crl}
\begin{proof}
Let $T$ be the submonoid of $M$ generated by the conjugates $sh_is^{-1}$ with $i\in[k]$ and $s\in\sym_n$. Since $M$ contains $\sym_n$ and, as explained in the proof of Theorem~\ref{010}, each $h_i\in M$, every conjugate $sh_is^{-1}$ belongs to $M$. Furthermore, since $sh_is^{-1}\in\P_n$, we obtain $sh_is^{-1}\in\P_n^M$ for all $i\in[k]$ and $s\in\sym_n$. Hence, $T\subseteq\P_n^M$. Conversely, as shown in the second part of the proof of Theorem~\ref{010}, any non-identity $e\in\P_n^M$ can be written as a product of conjugates $u_jh_iu_j^{-1}$ with $u_j\in\sym_n$ and $i\in[k]$. Thus, every $e\in \P_n^M$ belongs to $T$, concluding that $\P_n^M=T$.
\end{proof}

For instance, the submonoid of $T\sym_4$ containing $\sym_4$ associated with the antichain $\{(3,1),(2,2)\}$ is generated by the simple transpositions $s_1,s_2,s_3$ together with the idempotents $h_1=(\{1,2,3\},\{4\})$ and $h_2=(\{1,2\},\{3,4\})$. Similarly, the submonoid $M$ of $T\sym_4$ containing $\sym_4$ associated with the antichain $\{(4)\}$ is generated by $s_1,s_2,s_3$ together with the idempotent $(\{1,2,3,4\})$. In this case, we have $\P_4^M=\{1, (\{1,2,3,4\})\}$ because the idempotent $(\{1,2,3,4\})$ is central in $T\sym_4$. See Figure~\ref{003}.

\begin{figure}[H]
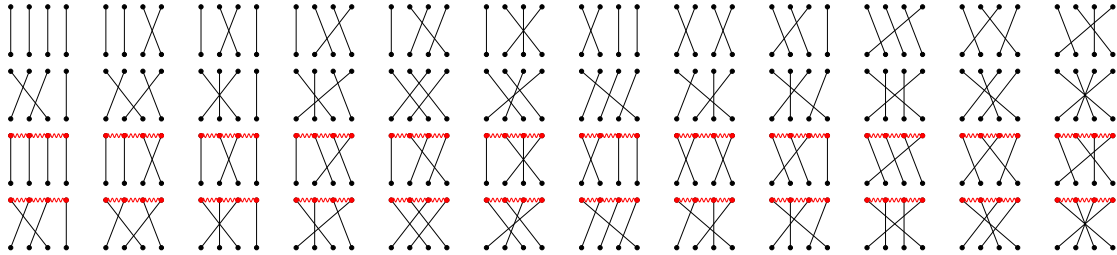
\figfou\caption{The elements of the submonoid of $T\sym_4$ containing $\sym_4$, associated to the antichain $\{(4)\}$.}\label{003}\end{figure}

\subsection{Concluding remarks}

A natural problem is to determine a presentation by generators and relations for the submonoids of $T\sym_n$ containing $\sym_n$. While the Coxeter relations and the trivial idempotency $h_i^2=h_i$ and commutativity relations $h_ih_j=h_jh_i$ are immediate, characterizing the complete set of merging relations remains an open problem for a general down-set $I$.

By Corollary~\ref{007}, the singular part of any such submonoid $M$, given by $M \setminus\sym_n=\bigsqcup_{\lambda\in I}J_\lambda$, is itself a subsemigroup of the universal singular ideal $T\sym_n\setminus\sym_n$ stratified by $\J$-classes. From a complementary perspective, recent work investigated this singular part, introducing a distinguished submonoid $\BR(\sym_n)$ closely related to $\sym_n$ whose non-identity elements lie entirely within $T\sym_n \setminus \sym_n$~\cite[Subsection~5.2]{ArEs26}. While the present paper characterizes the submonoids of $T\sym_n$ containing the group of units $\sym_n$, a natural and compelling parallel problem for future investigation is to classify the subsemigroups of the singular part $T\sym_n \setminus \sym_n$ that contain $\BR(\sym_n) \setminus \{1\}$.

\subsubsection*{Acknowledgements}
The author acknowledges the financial support of DIDULS/ULS, through the project PR2553853.

\bibliographystyle{plainurl}
\bibliography{../../../../Projects/LaTeX/bibtex.bib}

\begin{thebibliography}{10}

\bibitem{AiArJu23}
F.~Aicardi, D.~Arcis, and J.~Juyumaya.
\newblock Brauer and {J}ones tied monoids.
\newblock {\em J Pure Appl Algebra}, 227(1):107161, 1 2023.

\bibitem{AiArJu24}
F.~Aicardi, D.~Arcis, and J.~Juyumaya.
\newblock Ramified inverse and planar monoids.
\newblock {\em Mosc Math J}, 24(3):321--355, 9 2024.

\bibitem{AiJu00}
F.~Aicardi and J.~Juyumaya.
\newblock An algebra involving braids and ties.
\newblock ICTP Preprint IC/2000/179, 11 2000.

\bibitem{AiJu16b}
F.~Aicardi and J.~Juyumaya.
\newblock Markov trace on the algebra of braids and ties.
\newblock {\em Mosc Math J}, 16(3):397--431, 2016.

\bibitem{AiJu16}
F.~Aicardi and J.~Juyumaya.
\newblock Tied links.
\newblock {\em J Knot Theor Ramif}, 25(9):1641001, 2016.

\bibitem{AiJu21}
F.~Aicardi and J.~Juyumaya.
\newblock Tied links and invariants for singular links.
\newblock {\em Adv Math}, 381:107629, 4 2021.

\bibitem{ArEs26}
D.~Arcis and J.~Espinoza.
\newblock Tied--boxed algebras.
\newblock {\em J Algebra}, pages 112--159, 1 2026.

\bibitem{ArJu21}
D.~Arcis and J.~Juyumaya.
\newblock Tied monoids.
\newblock {\em Semigroup Forum}, 103(1--2):356--394, 10 2021.

\bibitem{PreArJu26}
D.~Arcis and J.~Juyumaya.
\newblock Party-{H}ecke algebras.
\newblock Preprint, 3 2026.
\newblock URL: \url{arxiv.org/abs/2603.19917}.

\bibitem{Moo896}
E.~Moore.
\newblock Concerning the abstract groups of order $k!$ and $\frac{1}{2}k!$
  holohedrically isomorphic with the symmetric and the alternating
  substitution--groups on $k$ letters.
\newblock {\em P Lond Math Soc}, 28(1):357--367, 11 1896.

\bibitem{OrSaSchZa22}
R.~Orellana, F.~Saliola, A.~Schilling, and M.~Zabrocki.
\newblock Plethysm and the algebra of uniform block permutations.
\newblock {\em Algebraic Combinatorics}, 5(5):1165--1203, 2022.

\bibitem{OrSaSchZa25}
R.~Orellana, F.~Saliola, A.~Schilling, and M.~Zabrocki.
\newblock The lattice of submonoids of the uniform block permutations
  containing the symmetric group.
\newblock {\em Semigroup Forum}, 110:405--421, 4 2025.

\bibitem{St97}
R.~Stanley.
\newblock {\em Enumerative Combinatorics, Volume 1}, volume~49 of {\em
  Cambridge Studies in Advanced Mathematics}.
\newblock Cambridge University Press, Cambridge, 1997.

\end{thebibliography}

\end{document}